\documentclass[10pt,leqno]{amsart}
\usepackage{amscd}
\usepackage{color}
\usepackage{amssymb}
\usepackage{amsfonts}
\usepackage{latexsym}
\usepackage{verbatim}

\theoremstyle{plain}
\newtheorem{theorem}{Theorem}[section]
\newtheorem{definition}[theorem]{Definition}
\newtheorem{lemma}[theorem]{Lemma}
\newtheorem{prop}[theorem]{Proposition}
\newtheorem{cor}[theorem]{Corollary}
\newtheorem{rem}[theorem]{Remark}
\newtheorem{ex}[theorem]{Example}
\newtheorem{problem}{Problem}

\renewcommand{\b}{\begin{equation}}
\newcommand{\e}{\end{equation}}
\newcommand{\g}{\mathfrak{g}}
\newcommand\C{{\mathbb C}}
\newcommand\R{{\mathbb R}}

\newcommand{\ov}[1]{\overline{ #1}}

\newcommand{\N}{\nabla}

\newcommand{\bk}{\overline{k}}

\newcommand{\br}{\overline{r}}

\title{Special Hermitian metrics  on  compact solvmanifolds}
\begin{document}
\dedicatory{Dedicated to Professor Paul Gauduchon on the occasion of his 70'th birthday.}
\author{Anna Fino  and Luigi Vezzoni}
\date{\today}
\subjclass[2000]{Primary 53C15 ; Secondary 53B15, 53C30}
\keywords{Hermitian metrics, symplectic forms, nilpotent Lie groups}
\address{Dipartimento di Matematica \lq\lq Giuseppe Peano\rq\rq \\ Universit\`a di Torino\\
Via Carlo Alberto 10\\
10123 Torino\\ Italy} 
 \email{annamaria.fino@unito.it}\email{luigi.vezzoni@unito.it}

\thanks{This work was supported by the project FIRB ``Geometria differenziale e teoria geometrica delle funzioni'',
 the project PRIN
\lq\lq  Variet\`a reali e complesse: geometria, topologia e analisi armonica" and by G.N.S.A.G.A. of I.N.d.A.M}

\begin{abstract} We review some constructions and properties of complex manifolds admitting pluriclosed and balanced metrics. We prove  that  for a  $6$-dimensional solvmanifold endowed with
an invariant complex structure $J$ having
holomorphically trivial canonical bundle   the pluriclosed flow has a long time solution for every invariant initial 
datum. Moreover,  we state a new conjecture about the existence of balanced and SKT metrics on compact complex manifolds. We show that the conjecture is true  for  nilmanifolds of dimension $6$ and $8$ and for $6$-dimensional solvmanifolds with holomorphically trivial canonical bundle.
\end{abstract}

\maketitle

\section{Introduction}

A very active field in  almost complex and  complex geometry is the seek of Hermitian metrics having special properties. A Riemannian metric $g$ on an almost complex manifold $(M,J)$ is called {\em Hermitian} if $g(JX,JY)=g(X,Y)$ for every vector fields $X$ and $Y$  on $M$. The pair $(g,J)$ is usually called an {\em almost Hermitian structure} (simply {\em Hermitian} when $J$ is integrable) and $\omega (X, Y) = g (X,JY)$ the {\em fundamental $2$-form}.  The pair $(g,J)$ specifies a ${\rm U}(n)$-structure whose intrinsic torsion can be identified with 
the covariant derivative  of  $ \omega$ with respect to the Levi-Civita connection of $g$. 
Therefore, in contrast to the K\"ahler case, when an almost Hermitian structure $(g,J)$ has  non-vanishing intrinsic torsion, the Levi-Civita connection of $g$ does not preserve $J$ and for this reason its role is often replaced by  other connections preserving $(g,J)$ but having torsion. A connection which preserves  $g$ and $J$  is usually called {\em Hermitian}. Fortunately,  the set of Hermitian connections always contains some canonical elements, distinguished by some properties of the torsion \cite{Gaudbumi}.  One of these connections was defined by Chern in \cite{chern}  to compute the representatives of the 
Chern classes, while another canonical connection was introduced by Bismut in \cite{B} to obtain an index theorem for non-K\"ahler manifolds. The choice of an Hermitian connection specifies a geometry on the manifold strictly linked to a special class of Hermitian metrics. The present paper focuses on some of these classes, specially on SKT   metrics (defined by the condition $\partial \bar \partial \omega=0$) and on balanced metrics (which are characterised by the condition $d^*\omega=0$). In particular we study the geometry of special metrics on {\em nilmanifolds} and {\em solvmanifolds} endowed with an {\em invariant complex structure}. By  a {\em nilmanifold}, we mean a compact manifold obtained as quotient of a simply-connected nilpotent Lie group $G$  by a lattice $\Gamma$. The definition of solvmanifold is the same, but the Lie group $G$ is taken {\em solvable} instead of nilpotent. In both cases,  by {\em invariant complex structure} we mean a complex structure which comes from   a   left-invariant complex structure on the Lie group $G$.   

For SKT metrics we take into account the so-called pluriclosed flow introduced by Streets and Tian in \cite{streets-tian2}, reviewing some our previous results about the long-time existence on nilmanifolds and providing new results on solvmanifolds with holomorphically trivial canonical bundle. 
It is well known that in any real dimension $2n$ the canonical bundle of a nilmanifold
$\Gamma \backslash G$ endowed with an invariant complex structure is holomorphically trivial.  Indeed, Salamon showed  in  \cite{salamon}  the existence
of a closed non-zero invariant $(n, 0)$-form. 
In \cite{FU}Ê the 6-dimensional solvmanifolds  $\Gamma \backslash G$ admitting an invariant
 complex structure with holomorphically trivial canonical bundle are determined.  The corresponding Lie algebras  $\frak g$ of the Lie groups $G$ and  the complex
structures are classified up to isomorphism, and the existence of  special Hermitian metrics (like for instance SKT  and balanced metrics) is studied. 
In this paper, by using these classifications, we show that if  $(M=\Gamma\backslash G,J)$  is  a $6$-dimensional solvmanifold endowed with
an invariant complex structure $J$ having
holomorphically trivial canonical bundle, then the pluriclosed flow has a long time solution for every invariant initial 
datum $g_o$.  

In the last part of the paper we study the  existence of two different types of Hermitian metrics on a  fixed complex manifold. We conjecture that in the non-K\"ahler compact case it  is never possible to find an SKT metric and also a balanced one. We prove the conjecture for nilmanifolds  of dimension $6$ and $8$ and for  $6$-dimensional solvmanifolds having holomorphically trivial canonical bundle.

%\bigbreak\noindent{\it Acknowledgments.} 

\section{Canonical Connections in Hermitian Geometry}
The definition of {\em canonical connection} was introduced by 
Gauduchon in \cite{Gaudbumi} in order to unify some special Hermitian connections described in literature in different contexts. 
Roughly speaking,  an Hermitian connection $\nabla$  on an almost Hermitian manifold $(M, J, g)$ is called  {\em canonical} if a component of its torsion tensor $T$ vanishes. In order to explain precisely the definition, we have to introduce some notation. 

 The complex structure $J$ induces a splitting of the complexified tangent bundle $T_{\C}M=TM\otimes \C$ in $T_{\C}M=T^{1,0}\oplus T^{0,1}$. Consequently,  the bundle $\Lambda^{p}_{\C}M$ of complex $p$-forms on $M$ splits as 
$$
\Lambda^{p}_{\C}M=\bigoplus_{r+s=p}\Lambda^{r,s}
$$
and the differential operator can be written as $d=A+\partial+\bar\partial+\bar A$, accordingly to the above splitting. The components $A$ and $\bar A$  vanish if and only if $J$ is integrable, i.e. if and only if  the Nijenhuis tensor $$N(X,Y)=[JX,JY]-[X,Y]-J([JX,Y]+[X,JY])$$
vanishes and,  in this case,  $d$ reduces to $d=\partial+\bar \partial$. Furthermore\,  the bundle $\Lambda^2(TM)=\Lambda^{2}M\otimes TM$ 
of real $2$-forms taking value in the tangent bundle $TM$ inherits the splitting  
$$
\Lambda^2(TM)=\Lambda^{2,0}(TM)\oplus \Lambda^{1,1}(TM) \oplus \Lambda^{0,2}(TM)
$$
where
$$
\begin{aligned}
& \Lambda^{2,0}(TM)=\{B\in\Lambda^2(TM)\,\,:\,\,B(JX,Y)=JB(X,Y) \}\,,\\
\vspace{0.1cm}
& \Lambda^{1,1}(TM)=\{B\in\Lambda^2(TM)\,\,:\,\,B(JX,JY)=B(X,Y) \}\,,\\
\vspace{0.1cm}
& \Lambda^{0,2}(TM)=\{B\in\Lambda^2(TM)\,\,:\,\,B(JX,Y)=-JB(X,Y) \}\,.
\end{aligned}
$$
The bundle $\Lambda^{1,1}(TM)$ can be further  decomposed  as
$$
\Lambda^{1,1}(TM)=\Lambda^{1,1}_b(TM)\oplus \Lambda^{1,1}_c(TM)
$$
where the projection $B_b$ and $B_c$ of $B$ onto $\Lambda^{1,1}_b(TM)$ and $\Lambda^{1,1}_c(TM)$ 
are respectively given by 
$$
\begin{aligned}
&2g(B_b(X,Y),Z)=(g(B(X,Y),Z)-g(B(Z,X),Y)-g(B(Y,Z),X))\,,\\
&2g(B_c(X,Y),Z)= (g(B(X,Y),Z)+g(B(Z,X),Y)+g(B(Y,Z),X))\,.
\end{aligned}
$$ 
If $\nabla$ is an Hermitian connection, then its torsion tensor $T$ is a section of $\Lambda^{2}(TM)$ and its $(1,1)$-component splits accordingly to the obove decomposition as 
$$
T^{1,1}=T^{1,1}_c+T^{1,1}_b\,. 
$$

\begin{definition}
An Hermitian connection $\nabla$  on an almost Hermitian manifold $(M, J, g)$ is {\em canonical}  if its torsion $T$ satisfies $T^{1,1}_b=0$. 
\end{definition}
\noindent Let us consider the following notation:
\begin{itemize}
\item  $J$ extends to $r$-forms by
$J\alpha(X_1,\dots,X_r)=(-1)^r \alpha(JX_1,\dots,JX_r)$ and we denote by $d^c$ the operator $d^c=(-1)^rJdJ$ on $r$-forms.
\item The bundle $\Lambda^{3}M$ of real $3$-forms splits as
$$
\Lambda^{3}M=\Lambda^{+}M\oplus\Lambda^{-}M\,,
$$
where $\Lambda^{+}M=(\Lambda^{2,1}M\oplus\Lambda^{1,2}M)\cap \Lambda^{3}M$ and
$\Lambda^{-}M=(\Lambda^{3,0}M\oplus\Lambda^{0,3}M)\cap \Lambda^{3}M$. Given a $3$-form $\gamma$ we denote by
$\gamma^+$ and $\gamma^-$ the projection onto $\Lambda^{+}M$ and $\Lambda^{-}M$, respectively. 
\end{itemize}
\begin{theorem}[Gauduchon \cite{Gaudbumi}]
Every canonical connection $\nabla$ satisfies 
\begin{equation}
\label{nablat}
\begin{aligned}
g(\N_{X}Y,Z) = &\,g(D_{X}Y,Z)+\frac{t-1}{4}(d^c\omega)^+(X,Y,Z)+\frac{t+1}{4}(d^c\omega)^+(X, JY, JZ)-g(X,N(Y,Z))+\\
                 &\,\frac12 (d^c \omega)^- (X,Y,Z)\,,
\end{aligned}
\end{equation}
for some $t\in\R$, where $D$ is the Levi-Civita connection of $g$  and $N$ denotes the Nijenhuis tensor of $J.$
\end{theorem}
When $J$ is integrable (i.e. when its Nijenhuis tensor vanishes), formula \eqref{nablat} simplifies to 
$$
g(\N^t_{X}Y,Z) = g(D_{X}Y,Z)+\frac{t-1}{4}(d^c\omega)(X,Y,Z)+\frac{t+1}{4}(d^c\omega)(X, JY, JZ)\,.
$$
In particular in the K\"ahler case the family reduces to a single connection. More generally, in the almost-K\"ahler   case we have a unique canonical connection (which is not the Levi-Civita connection  in this case) and in the co-symplectic case (i.e. when $\omega$ is co-closed) all the canonical connections have the same Ricci form (see e.g.  \cite[Corollary 3.3]{luigiproc}). Indeed, in general  the Ricci form of $\nabla^t$  is always a closed form which can be locally written as
the derivative of the 1-form $\theta^t (X) = \sum_{r = 1}^n g(\nabla^t Z_r, \overline{Z}_r)$, where $\{Z_r \}$ is a local unitary frame. If $\omega$ is co-closed, one can show that  $\theta^1= \theta^{-1}$.

\section{Strong K\"ahler metrics with torsion}
Let $(M,J,g)$ be an Hermitian manifold. For $t=-1$ the family \eqref{nablat} specifies the so-called {\em Bismut connection}  $\nabla^B$. This connection is the unique Hermitian connection whose torsion $T^B$, regarded as a $(0,3)$-tensor via the Hermitian metric $g$, is skew-symmetric and it was introduced  by Bismut in \cite{B}  to prove  an index theorem for non-K\"ahler Hermitian manifolds. Almost Hermitian structures admitting an  Hermitian connection with skew-symmetric torsion are characterized in \cite{FI}. In particular Theorem 10.1 in \cite{FI} implies that in the strictly almost-K\"ahler case such connections cannot exist. 

Usually the $3$-form induced by $T^B$ is denoted by $c$, i.e.
$$
c(X,Y,Z)=g(T^B (X,Y),Z)\,.
$$
\begin{prop}
The $3$-form $c$ is closed if and only if the fundamental form of the  Hermitian metric $g$ satisfies 
\begin{equation}\label{SKT}
\partial\bar\partial\omega=0\,.
\end{equation}
\end{prop}
Hermitian metrics whose fundamental form satisfies \eqref{SKT} are usually called {\em strong K\"ahler with torsion} (SKT in short)   or {\em pluriclosed} (see e.g. \cite{FTsurvey} for a survey on SKT metrics).  
Such metrics have applications in type II string
theory and in $2$-dimensional supersymmetric $\sigma$-models \cite{GHR,strominger} and have relations with generalized K\"ahler structures (see for instance \cite{AS,FT3,Gualtieri,Hitchin}). Every compact  complex surface admits an SKT structure in view of the following 
\begin{theorem}[Gauduchon \cite{Gaustandard}]
Let $(M^n,J,g)$ be a compact Hermitian manifold of  complex dimension $n$. Then there exists in the conformal class of $g$ a unique Hermitian structure 
$\tilde g$ (up to homotheties) whose fundamental form $\tilde \omega$ satisfies 
$$
\partial\bar \partial \tilde \omega^{n-1}=0\,.
$$
\end{theorem}
In higher dimensions  the existence of an SKT structure is not always guaranteed. For instance, SKT metrics cannot exist  on  non-K\"ahler  twistor spaces of  compact, anti-self-dual Riemannian manifolds \cite{verb}. 

Examples of SKT manifolds are provided by $6$-dimensional nilmanifolds in view of the following
\begin{theorem}[Fino-Parton-Salamon \cite{finosalamonparton}] \label{finosalamonparton}
Let $M=\Gamma\backslash G$ be a $6$-dimensional nilmanifold  endowed with an
invariant complex structure $J$. Then the SKT condition is satisfied by either
all invariant Hermitian metrics $g$ or by none. Indeed, $(M, J)$  admits a  SKT metric  if and
only if  the lie algebra $\frak g$ of $G$  has a basis $(\alpha^i)$ of $(1,0)$-forms such that
\begin{equation}
\begin{cases}
d\alpha^1=0\\
d\alpha^2=0\\
d\alpha^3=A\alpha^{\bar 12}+
B\alpha^{\bar 22}+C\alpha^{1\bar 1}+D\alpha^{1\bar 2}+ E\alpha^{12}
\end{cases}
\end{equation} 
where $A,B,C,D,E$ are
complex numbers satisfying the condition 
\begin{equation*}
|A|^2+|D|^2+|E|^2+2{\rm Re}(\bar BC)=0.
\end{equation*}
\end{theorem}

More in general,  by  \cite{EFV}  if a nilmanifold $M$ endowed with an invariant complex structure $J$ admits an SKT metric, then $M$  is at most 2-step. As a consequence  a  classification of  8-dimensional nilmanifolds endowed with an invariant complex structure admitting an SKT metric is given  in \cite{EFV}.

Other examples of SKT metrics on compact manifolds are given by the connected sum of products of spheres in view of the following  

\begin{theorem}[Grantcharov-Grantcharov-Poon \cite{GGP}] 
For any positive integer $k \geq 1$, the manifold $M_k=(k - 1)(S^2 \times S^4)\sharp k(S^3 \times S^3)$
admits an SKT structure.
\end{theorem}
Moreover, examples of SKT manifolds can be constructed via complex  blow-up construction, as shown by the following

\begin{theorem}[Fino-Tomassini \cite{FTadv}]  The complex  blow-up of an SKT manifold M at a point or along a compact complex submanifold admits an SKT metric.
\end{theorem}

 Recently, Cavalcanti in \cite{Cavalcanti}  used generalized complex geometry to   study  SKT manifolds  and more generally manifolds with special holonomy with respect to a metric
connection with closed skew-symmetric torsion.  He  developed Hodge theory on such manifolds showing
how the reduction of the holonomy group causes a decomposition of the twisted cohomology.  In particular, he proved that the only Calabi-Eckman manifolds  admitting an SKT structures are $S^1 \times S^1, S^1 \times S^3$ and $S^3 \times S^3$.

\section{Geometric flows of Hermitian metrics}
In \cite{Cao} Cao obtained a new proof of the Calabi-Yau theorem via the Ricci flow. A key observation in  his paper is that  for a compact K\"ahler manifold $(M, J, g_o)$,  the solution $g(t)$ to the Ricci flow 
\begin{equation}\label{RF}
\partial_tg(t)=-{\rm Rc}(g(t))\,,\quad g(0)=g_o
\end{equation}
is still K\"ahlerian for every $t$ where it is defined. Moreover,  \eqref{RF} can be regarded asa flow of  $2$-forms by identifying a K\"ahler metric with its fundamental form and the Ricci tensor with the Ricci form, i.e. 
$$
\partial_t\omega(t)=-\rho(\omega(t))\,,\quad \omega(0)=\omega_o\,.
$$
It turns out that, in contrast to the Riemannian case, in the K\"ahler setting equation \eqref{RF} is parabolic in a strong sense and then the short-time existence is  ensured by the standard parabolic theory (see e.g. \cite{Aubin}). 
 Moreover,  Cao proved the following result which implies the statement of the Calabi-Yau theorem. 
\begin{theorem}[Cao \cite{Cao}]  Let   $(M, J, g_o)$  be a compact K\"ahler manifold  and let $T$ be a representative of $2\pi c_1(M,J)$.  Then the maximal solution  to the {\em K\"ahler-Ricci flow} 
$$
\partial_t\omega(t)=-\rho(\omega(t))+T\,,\quad \omega(0)=\omega_o
$$
is defined  for evert $t  \in [0,\infty)$ and converges to the K\"ahler form of a K\"ahler metric having $T$ as Ricci form.  
\end{theorem}

In \cite{streets-tian1} Streets and Tian introduced a generalization of the K\"ahler-Ricci flow to the the Hermitian case with torsion.  The basic idea in \cite{streets-tian1} is to use the Chern connection to construct a parabolic flow of Hermitian metrics instead of the Levi-Civita connection. We briefly describe the {\em Hermitian curvature flow} introduced in \cite{streets-tian1}.
Let $(M,J,g)$ be an Hermitian manifold with Chern connection $\nabla^c$. Let $R^c$ be the curvature of $\nabla^c$ and $S(g)$ be the $(1,1)$-tensor  given  by 
$$
S(g)_{k\bar r}:=g^{\bar j i}R^c_{i\bar jk\bar r}\,. 
$$
Then $g\mapsto S(g)$ defines an operator from the cone of Hermitian metrics on $(M,J)$ to $J$-invariant symmetric $2$-tensors on $M$ which is, in view of \cite{streets-tian1}, a quasi linear {\em elliptic} operator of the  second order. As such once an initial Hermitian metric $g_o$ is  fixed,  the geometric flow 
$$
\partial_t(g(t))=-S(g(t))\,,\quad g(0)=g_o
$$
has always a short-time solution. Furthermore, by adding to $-S(g)$ a tensor $Q(g)$ quadratic in the torsion of $\nabla^c$,  the modified curvature flow 
\begin{equation}\label{HCF}
\partial_t(g(t))=-S(g(t))+Q(g(t))\,,\quad g(0)=g_o
\end{equation}
turns out to be the gradient flow of the  functional 
$$
\mathbb{F}(g):=\frac{\int_M \left(s^c-\frac14\,|T^c|^2-\frac12 |w|^2\right)\,dV}{(\int_M dV)^{\frac{n-1}{n}}}
$$
acting on the space of the Hermitian metrics. Here   $dV=\frac{1}{n!}\omega^n$ is the volume form induced by the fundamental form $\omega$ of $g$, $T^c$ is the torsion of $\nabla^c$, $w$ is the $1$-form $w_i:=(T^c)_{ik}^{k}$ and the norms are computed with respect to the metric $g$. 
In \cite{streets-tian2} it is proved that the Hermitian curvature flow preserves the SKT condition in the sense that when the initial metric $g_o$ is SKT, then the solution $g(t)$ to 
\eqref{HCF} holds SKT for every $t$. Moreover in the SKT case the Hermitian curvature flow, regarded as a flow of $2$-forms,  reduces to 
\begin{equation}\label{PCF}
\partial_t\omega(t)=-[\rho^B(\omega(t))]^{1,1}\,,\quad \omega(0)=\omega_o\,,
\end{equation}
$\rho^B$ being the curvature form of the Bismut connection and the superscript $(1,1)$ denoting the projection onto $\Lambda^{1,1}$. 

\begin{rem}
{\em In \cite{luigiflow} it is observed that the Hermitian curvature flow can be generalised to the almost-Hermitian case by adding an extra term to the definition of $Q$. Such a generalisation does not preserve the almost-K\"ahler condition and a curvature flow preserving almost-K\"ahler structures is provided in \cite{streets-tian3}. The definition of this last flow, called {\em symplectic curvature flow}, is highly non trivial and the almost complex structure evolves with the metric. The symplectic curvature flow has been generalised to the almost-Hermitian non-symplectic case in \cite{smith,dai}}. 
\end{rem}

\subsection{Solutions to the pluriclosed flow on homogeneous spaces} 
 The behaviour of the solutions to the pluriclosed flow on nilmanifolds is analysed in \cite{EFV2}.  The idea in \cite{EFV2} consists in adapting the argument used by Lauret in \cite{lauret} to study the Ricci flow regarding the pluriclosed flow in the invariant case as a flow of brackets instead of invariant metrics.

Let $M=\Gamma\backslash G$ be a nilmanifold endowed with an invariant SKT structure $(J,\omega_o)$ and let $\omega(t)$ be the maximal solution to the pluriclosed flow \eqref{PCF}. 
Since \eqref{PCF} is invariant by biholomorphisms and $\omega_o$ is invariant, then $\omega(t)$ is still invariant for every $t$ and \eqref{PCF} reduces to an ODE.  Therefore the maximal solution is defined for $t\in (-\epsilon,T)$, for some $\epsilon,T>0$. Moreover, $\omega(t)$ can be regarded as a flow of metrics on the Lie algebra $(\g,\mu_o)$ of $G$ (here $\mu_o$  denotes the Lie bracket). This remark allows us to work in an algebraic fashion. Fix a $(1,0)$-basis of the Lie algebra $(\mathfrak g,J)$ and denote by 
$\mu_{AB}^C$ the components of the bracket $\mu_o$ (the capital letters run into the index set $\{1,\dots,n,\bar 1,\dots,\bar n\}$).  Then a direct computation yields that the component of the $(1,1)$-part of the Bismut Ricci  form $\rho^B$ of an Hermitian metric on $(\g,J)$ reads,  in terms of the basis of $\g$,  as 
\begin{equation}\label{rho1}
\rho^B_{i\bar j}=\,- \mu_{i\bar j }^a\mu_{ar}^{r}+ \mu_{i\bar j }^ag^{\bk r}\mu_{r\bk}^{\bar l} g_{a\bar l}
                 +\mu_{i\bar j }^{\bar b}\mu_{\bar b \bar r}^{\bar r}+\mu_{i\bar j }^{\bar b}g^{k \br}\mu_{k \bar r }^{ l} g_{l\bar b}
\end{equation}
and  \eqref{PCF} can be written as
\begin{equation}\label{flowSTalgebra}
\begin{cases}
\frac{d}{dt}g_{i\bar j}=\,\, \mu_{i\bar j }^a\mu_{ar}^{r}-\mu_{i\bar j }^ag^{\bk r}\mu_{r\bk}^{\bar l} g_{a\bar l}
                 -\mu_{i\bar j }^{\bar b}\mu_{\bar b \bar r}^{\bar r}-\mu_{i\bar j }^{\bar b}g^{k \br}\mu_{k \bar r }^{ l} g_{l\bar b}\\
                 g_{i\bar j}(0)=(g_0)_{i\bar j}\,.
\end{cases}
\end{equation}

We can identify $\g$ with $\R^{2n}$, $(J,\omega_o)$ with the standard Hermitian structure in $\R^{2n}$ and $\mu_o$ with a Lie  bracket in $\R^{2n}$. Since  by hypothesis  $J$ is a complex structure, its Nijenhuis  tensor vanishes and therefore 
\b\label{intwithmu}
\mu_o(JX,JY)=\mu_o(JX,Y)+\mu_o(X,JY)+\mu_o(X,Y),
\e
for every $X,Y\in \R^{2n}$. On the other hand, $\omega(t)$ solves the ODE
\begin{equation}\label{PCFalg}
\frac{d}{dt}\omega=-(\rho^B_{\mu_o})^{1,1}\,,\quad \omega(0)=\omega_o,
\end{equation}
where given a Lie  bracket $\mu$ satisfying \eqref{intwithmu}, $\rho^B_{\mu}$ is the skew-symmetric form 
$$
\rho^B_{\mu}(X,Y)=i \sum_{r=1}^{n}\Big(\langle\mu(\mu(X,Y),Z_r),Z_{\br}\rangle-\langle\mu(Z_r,Z_{\br}),\mu(X,Y)\rangle\Big)\,,
$$
$\{Z_r\}$ is the standard unitary basis on $(\R^{2n},J_0)$, $\langle\cdot,\cdot\rangle$  is the standard Euclidean metric and the superscript $(1,1)$ denotes the projection onto $\Lambda^{1,1}$.   Given a Lie bracket $\mu$ satisfying \eqref{intwithmu} we denote by $P_\mu$ the complex isomorphism
$$
\omega_o(P_{\mu}X,Y)=(\rho^{B}_{\mu})^{1,1}(\omega)\,.
$$
induced by $\mu$ and $\omega_o$. Then we consider the flow of  nilpotent  Lie brackets 
\begin{equation}\label{bracketflow}
\frac{d}{dt}\mu=\frac12\, \delta_{\mu}(P_{\mu})\,,\quad \mu(0)=\mu_o
\end{equation}
where  we set 
$$
\delta_{\mu}(\alpha)=\mu(\alpha\cdot,\cdot)+\mu(\cdot,\alpha\cdot)-\alpha\mu(\cdot,\cdot)\,,
$$
for a given  endomorphism $\alpha$. The following results hold

\begin{theorem}[Enrietti-Fino-Vezzoni, \cite{EFV2}]\label{EFVflow}
The bracket flow \eqref{bracketflow} has always a solution $\mu(t)$ for $t\in (-\epsilon,\infty)$ converging to the trivial bracket on $\R^{2n}$. Moreover all the brackets  $\mu(t)'$s have the same center and there exists a solution  $h(t)\in {\rm GL}(n,\C)$, for  $t\in (-\epsilon,\infty)$, to the the flow 
$$
\frac{d}{dt}h(t)=-P_{\mu(t)}h(t)\,,\quad h(0)=h_o\,. 
$$ 
Finally $\omega(t)=h(t)^*(\omega_o)$ solves \eqref{PCFalg}. 
\end{theorem}
\begin{cor}
Let $(M=\Gamma\backslash G,J,\omega_0)$ be a {\em nilmanifold}  endowed with an invariant SKT structure. Then the maximal solution $\omega(t)$ to the pluriclosed flow is defined in $(-\epsilon,\infty)$, where $\epsilon$ is a suitable positive real  number.
\end{cor}

\begin{ex}[The solution on the Kodaira-Thurston manifold]{\em 
In dimension $4$ the unique nilpotent Lie algebra (up to isomorphisms) carrying an SKT structure
is  $\mathfrak{h}_3 \oplus\R$, where $\mathfrak{h}_3$ is the  Lie algebra of the 3-dimensional real Heisenberg Lie group $H_3 (\R)$  given by
$$
 H_3 (\R) = \left \{  \left (    \begin{array}{ccc} 1&x&z\\ 0&1&y\\ 0&0&1  \end{array}  \right ), \quad x, y, z \in \R  \right \}.
$$
The {\em Kodaira-Thurston surface} is the compact quotient of the  simply-connected Lie group $H_3 (\R) \times \R$ by the lattice $\Gamma\times \mathbb Z $, where $\Gamma$ is the lattice in $H_3 (\R)$ whose elements  are matrices with integer entries.
The Lie algebra  $\g=\mathfrak h_3\oplus\R$  has structure equations $(0,0,0,12)$, where with this notation we mean that there exists a basis of 1-forms $\{ e^i \}$ such that
$$
d e^i = 0,\,\,\, i = 1,2,3, \quad de^4 = e^1 \wedge e^2.
$$
 Let $J$ be the complex structure on $\g$  given by
$$
Je_1=-e_2,\,\quad  Je_3=-e_4\,.
$$
Then
$$
Z_1=\frac12\, (e_1+ie_2)\,,\quad Z_2=\frac12\,(e_3+ie_4)
$$
is a complex basis of type $(1,0)$ of $(\g,J)$. Let $\{\zeta^1,\zeta^2\}$ be  its dual frame.
Every Hermitian inner product  $g$ on $(\g,J)$ can be written as
$$
g=x\zeta^{1}\zeta^{\bar 1}+y\zeta^{2}\zeta^{\bar 2}+z\zeta^{1}\zeta^{\bar 2}+\bar z \zeta^{2}\zeta^{\bar 1}\,.
$$
where $x,y\in \R$, $z\in \C$ satisfy $xy-|z|^2>0$ and  it is SKT.
Since
$$
\mu(Z_1,Z_{\bar 1})=-\frac12(Z_2-Z_{\bar 2})
$$
is the only non-vanishing bracket, the Bismut Ricci form $\rho^B$ of $g$ has component only along $\zeta^{1\bar 1}$, i.e. 
$
\rho^{B}=-i\rho^{B}_{1\bar 1}\, \zeta^{1\bar 1}
$
and a direct computation yields
$$
\rho^{B}_{1\bar 1}= -\frac{y^2}{2\left(xy-|z|^2\right)}\,.
$$
Therefore in this case the pluriclosed flow with initial condition $\omega_o=-i\zeta^{1\bar 1}-i\zeta^{2\bar2 }$ reduces to
\begin{equation}
\label{KT}
\dot{x}=\frac{y^2}{2\left(xy-|z|^2\right)}\,,\quad y\equiv 0\,,\quad z\equiv 0\,,\quad 
x(0)=1
\end{equation}
and its maximal solution is  
%and  the solution to \eqref{flowRn} with
%$$
%\omega_0=-ix_0\zeta^{1\bar 1}-iy_0\zeta^{2\bar 2}-iz_0\zeta^{1\bar 2}-i\bar z_0\zeta^{2\bar 1}
%$$
%is
%$$
%\omega(t)=-ix(t)\zeta^{1\bar 1}-iy_0\zeta^{2\bar 2}-iz_0\zeta^{1\bar 2}- i \bar z_0\zeta^{2\bar 1}
%$$
%where
%$$
%x(t)=\frac{1}{y_0}\left(\sqrt{y_0^2 t+(x_0y_0-|z_0|^2)^2}+|z_0|^2\right)\,.
%$$
%For instance if we start from the standard SKT structure
%$$
%\omega_0=-i\zeta^{1\bar 1}-i\zeta^{2\bar 2}
%$$
%we get
$$
\omega(t)=-i\sqrt{t+1}\,\zeta^{1\bar 1}-i\zeta^{2\bar 2}
$$
for $t\in (-1,\infty)$.
From the viewpoint of the bracket flow, the initial bracket takes the following expression 
$$
\mu_0=-\frac12\, \zeta^1\wedge\zeta^{\bar 1}\otimes Z_2+\frac12\, \zeta^1\wedge\zeta^{\bar 1}\otimes Z_{\bar 2}\,.
$$
Since the bracket flow  preserves the center, we look for a solution $\mu$ to \eqref{bracketflow}  taking value only at $(Z_1,Z_{\bar 1})$, i.e.
$$
\mu=\mu_{1\bar 1}^2\, \zeta^1\wedge\zeta^{\bar 1}\otimes Z_2+\mu_{1\bar 1}^{\bar 2}\, \zeta^1\wedge\zeta^{\bar 1}\otimes Z_{\bar 2}\,.
$$
For such a bracket we have
$$
\rho^B_{\mu}=-2i\,|\mu_{1\bar 1}^2|^2\, \zeta^1\wedge\zeta^{\bar 1}
$$
and
$$
P_{\mu}=-2\,|\mu_{1\bar 1}^2|^2\, \zeta^1\otimes Z_{ 1}+2\,|\mu_{1\bar 1}^2|^2\, \zeta^{\bar 1}\otimes Z_{\bar 1}\,.
$$
Therefore
$$
\delta_{\mu}(P_{\mu})(Z_1,Z_{\bar 1})=2\mu(P_{\mu}(Z_1),Z_{\bar 1})=-4 |\mu_{1\bar 1}^2|^2 \mu(Z_1,Z_{\bar 1})
$$
and the corresponding bracket flow equation is
\begin{equation}\label{z}
\dot{z}=-2 |z|^2\,z\,,\quad 
z(0)=-\frac12
\end{equation}
where $z=\mu_{1\bar 1}^2$.  Since \eqref{z} has as solution the real function
$$
z(t)=-\frac{1}{2\sqrt{t+1}}
$$
the solution to the bracket flow is 
$$
\mu(t)=-\frac{1}{2\sqrt{t+1}}\, \zeta^1\wedge\zeta^{\bar 1}\otimes Z_2-\frac{1}{2\sqrt{t+1}}\, \zeta^1\wedge\zeta^{\bar 1}\otimes Z_{\bar 2}\,.
$$
which is defined in $(-1,\infty)$ and converges to $0$ for $t\rightarrow \infty$, accordingly to Theorem \ref{EFVflow}. 
}
\end{ex}

In \cite{lauret3} Lauret describes a general approach to study curvature flows on almost Hermitian Lie groups. This description includes many flows of almost Hermitian  structures studied in the last years. 

For homogeneous complex surfaces Boling proves the following two results about the solutions to the pluriclosed flow.
\begin{theorem}[Boling \cite{boling}]
Let $\omega(t)$ be a locally homogeneous solution of the pluriclosed flow on a
compact
complex surface which exists on a maximal time interval $[0,T)$. 
If $T<\infty$ then the complex surface is rational or ruled. If $T=\infty$ and
the manifold is a Hopf surface, the evolving metric converges exponentially fast
to a canonical form unique up to homothety. Otherwise, there is a blowdown limit
\[\tilde g_\infty(t)=\underset{s\rightarrow\infty}\lim s^{-1}\tilde g(st)\] of
the
induced metric on the universal cover which is an expanding soliton in the sense
that $\tilde g(t) = t \tilde g(1)$ up to automorphism.
\end{theorem}

\begin{theorem}[Boling \cite{boling}]
Let $\omega(t)$ be a locally homogeneous solution of pluriclosed flow on a
compact
complex surface $(M,J)$ which exists on the interval $[0,\infty)$ and suppose
that $(M,J)$ is not a Hopf surface. Let $\hat{\omega}(t)=\frac{\omega(t)}{t}$. Then 
\begin{enumerate}
\item[1.]If the surface is a torus, hyperelliptic, or a Kodaira surface, then the family
$(M,\hat{g}(t))$ converges as $t\rightarrow\infty$ to a point in the
Gromov-Hausdorff sense.
\item[2.] If the surface is an Inoue surface, then the family $(M,\hat{\omega}(t))$
converges as $t\rightarrow\infty$ to a circle in the Gromov-Hausdorff sense and
moreover the length of this circle depends only on the complex structure of the
surface.
\item[3.] If the surface is a properly elliptic surface where the genus of the base
curve is at least 2, then the family $(M,\hat{g}(t))$ converges as
$t\rightarrow\infty$ to the base curve with a metric of constant curvature.
\item[4.] If the surface is of general type, then the family $(M,\hat{\omega}(t))$ converges
as $t\rightarrow\infty$ to a product of K\"ahler-Einstein metrics on $M$.
\end{enumerate}
\end{theorem}

In \cite{lauret2} Lauret studies the Ricci flow on homogeneous manifolds using the bracket flow argument. We think  that an analogue approach could  also give insights for the pluriclosed flow. This will be the subject of a future work. 

\subsection{The pluriclosed flow on solvmanifolds with holomorphically trivial canonical bundle}
The aim of this section is to prove the following  
\begin{theorem}\label{lte}
Let $(M=\Gamma\backslash G,J)$ be a $6$-dimensional solvmanifold endowed with
an invariant complex structure $J$ having
holomorphically trivial canonical bundle. Then the pluriclosed flow has a long time solution for every invariant initial 
datum $g_o$.  
\end{theorem}
\begin{proof}  If $G$ is non nilpotent,  by \cite[Theorem 4.1]{FU}
$(M,J)$ has an SKT metric if and only if   the Lie algebra $\frak g$ of $G$  is either isomorphic to $\frak g_2^0=(e^{25},-e^{15}, e^{45},-e^{35}, 0, 0),$ or $\frak g_4=(e^{23},-e^{36}, e^{26},-e^{56}, e^{46}, 0)$.

The solvable Lie algebra $\frak g_2^0$ has, up to equivalence,  only  one  complex structure $J$  defined by 
the structure equations
\b\label{streqs}
d \alpha^1 = i  (\alpha^{13} + \alpha^{1 \overline 3}), \, d \alpha^2 = - i \alpha^{23} - i  \alpha^{2 \overline 3}, d \alpha^3 =0
\e
with respect to a suitable $(1,0)$-coframe $(\alpha^k)$ (see \cite[Proposition 3.3]{FU}). For the Lie algebra $\frak g_4$  by Proposition 3.6 in \cite{FU} any  complex structure $J$ with  a closed $(3,0)$-form  on $\frak g_4$  is equivalent to  one of  the  complex structure  $J_{\pm}$  given by  
\b\label{streqs2}
d \alpha^1 = i \alpha^{13} + i \alpha^{1 \overline 3}, \, d \alpha^2 = - i \alpha^{23} - i \alpha^{2 \overline 3}, \, d \alpha_3 = \pm \alpha^{1 \overline 1}.
\e
In both cases an invariant metric $g$ is SKT if and only if its component $g_{1\bar 2}$ (with respect to the coframe $(\alpha^k)$)   vanishes.  

We first prove the theorem when the Lie algebra of $G$ is $\frak g_2^0$.  
Equations \eqref{streqs} read in terms of the bracket $\mu$ and the dual frame $(Z_k)$ to $(\alpha^k)$ as 
$$
\mu_{13}=-iZ_1,\quad \mu_{1\bar 3}=-iZ_1\,,\quad \mu_{23}=iZ_2,\quad \mu_{2\bar 3}=iZ_2\,. 
$$
Therefore, given an Hermitian  metric $g$ on $(\g,J)$, the  only non-vanishing components of $(\rho^B)^{1,1}$ are
$$
\rho^B_{1\bar 3}=-i g^{\bar k r}g_{1\bar l}\mu_{r\bar k}^{\bar l}\,,\quad \rho^B_{2\bar 3}=i g^{\bar k r}g_{2\bar l}\mu_{r\bar k}^{\bar l}
$$
(and their conjugates).  In the SKT case we have $g_{1\bar 2}=0$ and the above formulas simplify to 
$$
\rho^B_{1\bar 3}=-i g^{\bar 1 3}g_{1\bar 1}\mu_{3\bar 1}^{\bar 1}= -g^{\bar 1 3}g_{1\bar 1} \,,\quad \rho^B_{2\bar 3}=i g^{\bar 2 3}g_{2\bar 2}\mu_{3\bar 2}^{\bar 2}=-g^{\bar 2 3}g_{2\bar 2}
$$

Now let $g_o$ be a fixed invariant SKT metric and let $g=g(t)$ be the maximal solution to the pluriclosed flow with initial datum $g_o$. Assume by contradiction that the time domain of $g$ is $[0,T)$, with $T<\infty$. 
The only components of $g$ which evolve with the pluriclosed flow are $g_{1\bar 3}$ and $g_{2\bar 3}$ (and their conjugates) while the other components remain constant. Denote $g_{1\bar 3}$ by $v$ and  $g_{2\bar 3}$ by $z$. Then since    
$$
g^{\bar 1 3}=-\frac{vg_{2\bar 2}}{\det(g_{k\bar r})},\quad g^{\bar 2 3}=-\frac{zg_{1\bar 1}}{\det(g_{k\bar r})}
$$
and 
$$
\det(g_{k\bar r})=g_{1\bar 1}g_{2\bar 2}g_{3\bar 3}-g_{1\bar 1}|z|^2-g_{2\bar 2}|v|^2
$$
we get 
$$
\rho^B_{1\bar 3}= \frac{vg_{1\bar 1}g_{2\bar 2}}{c-g_{1\bar 1}|z|^2+g_{2\bar 2}|v|^2}\,,
\quad \rho^B_{2\bar 3}=\frac{zg_{1\bar 1}g_{2\bar 2}}{c-g_{1\bar 1}|z|^2+g_{2\bar 2}|v|^2}
$$
where $c=g_{1\bar 1}g_{2\bar 2}g_{3\bar 3}$. 
In particular the pluriclosed flow equation reads in terms of $v$ and $z$ as 
$$
\begin{cases}
\dot v=\frac{vg_{1\bar 1}g_{2\bar 2}}{g_{1\bar 1}|z|^2+g_{2\bar 2}|v|^2-c}\\
\dot z=\frac{zg_{1\bar 1}g_{2\bar 2}}{g_{1\bar 1}|z|^2+g_{2\bar 2}|v|^2-c}\,.
\end{cases}
$$
Now a direct computation yields 
$$
\frac{d}{dt}|v|^2\leq 0\,,\quad \frac{d}{dt}|z|^2\leq 0
$$
and $|v|$, $|z|$ decrease along the flow. Hence $v$ and $z$ converge as $t\to T$ and so $g(t)$ converges to a in invariant tensor  $g(T)$ as $t\to T$. Since $|v|$ and $|z|$ decrease, $g(T)$ is still positive definite and we can extend the flow afterwards $T$, contradicting $T<\infty$. 

Now we consider the case in which the Lie algebra of 	$G$ is $\g_4$. In this case the proof goes more or less in the same way as for $\g_2^0$, but we have that to take into account that  also the component $(1,\bar 1)$ of the metric evolves. Structure equations \eqref{streqs2} read in terms of brackets as
$$
\mu_{13}=-iZ_1,\quad \mu_{1\bar 3}=-iZ_1\,,\quad \mu_{23}=iZ_2,\quad \mu_{2\bar 3}=iZ_2\,,\quad 
\mu_{1\bar1 }=\mp Z_3\pm Z_{\bar 3}
$$
and the Bismut form of a generic invariant SKT metric is given by the following relations 
$$
\rho^B_{1\bar 1}=
%\pm ig_{3\bar 1}g^{\bar 13}\mp ig_{1\bar 3}g^{\bar 3 1}-2g^{\bar 11}g_{3\bar 3}=
-2g^{\bar 11}g_{3\bar 3}, \quad 
\rho^B_{1\bar 3}=-g^{\bar 13}g_{1\bar 1}\mp ig^{\bar 11}g_{1\bar 3}, \quad 
\rho^B_{2\bar 3}=-g^{\bar 2 3}g_{2\bar 2}\pm i g^{\bar 1 1}g_{2\bar 3}.
$$
Let $g_o$ be an invariant fixed SKT metric and $g=g(t)$ be the maximal solution of the pluriclosed flow with initial condition $g_o$. Assume by contradiction that the time domain of $g$ is $[0,T)$, with $T<\infty$. 
In order simplify the notation we write $g_{1\bar 1}=x$, $g_{1\bar 3}=v$, $g_{2\bar 3}=z$. Then  
$$
g^{\bar 11}=\frac{g_{2\bar2}g_{3\bar 3}-|z|^2}{\det(g_{k\bar r})}\,,\quad 
g^{\bar 1 3}=-\frac{vg_{2\bar 2}}{\det(g_{k\bar r})},\quad g^{\bar 2 3}=-\frac{zg_{1\bar 1}}{\det(g_{k\bar r})}
$$
and 
$$
\det(g_{k\bar r})=x(g_{2\bar 2}g_{3\bar 3}-|z|^2)-g_{2\bar 2}|v|^2\,.
$$
Therefore 
$$
\begin{aligned}
&\rho^B_{1\bar 1}=-2g_{3\bar 3}\frac{g_{2\bar2}g_{3\bar 3}-|z|^2}{x(g_{2\bar 2}g_{3\bar 3}-|z|^2)-g_{2\bar 2}|v|^2}\\
&\rho^B_{1\bar 3}=v\frac{xg_{2\bar 2}\mp i(g_{2\bar2}g_{3\bar 3}-|z|^2)}{x(g_{2\bar 2}g_{3\bar 3}-|z|^2)-g_{2\bar 2}|v|^2}\\
&\rho^B_{2\bar 3}=z\frac{xg_{2\bar 2}\mp i(g_{2\bar2}g_{3\bar 3}-|z|^2)}{x(g_{2\bar 2}g_{3\bar 3}-|z|^2)-g_{2\bar 2}|v|^2}
\end{aligned}
$$
and the puriclosed flow read in terms of $x,v,z$ as 
$$\begin{aligned}
&\dot x=2g_{3\bar 3}\frac{g_{2\bar2}g_{3\bar 3}-|z|^2}{x(g_{2\bar 2}g_{3\bar 3}-|z|^2)-g_{2\bar 2}|v|^2}\\
&\dot v=-v\frac{xg_{2\bar 2}\mp i(g_{2\bar2}g_{3\bar 3}-|z|^2)}{x(g_{2\bar 2}g_{3\bar 3}-|z|^2)-g_{2\bar 2}|v|^2}\\
&\dot z=-z\frac{xg_{2\bar 2}\mp i(g_{2\bar2}g_{3\bar 3}-|z|^2)}{x(g_{2\bar 2}g_{3\bar 3}-|z|^2)-g_{2\bar 2}|v|^2}\,.
\end{aligned}
$$
Again we easily get 
$$
\frac{d}{dt}|v|^2\leq 0\,,\quad \frac{d}{dt}|z|^2\leq 0
$$
and that $x$ increase along the flow, while $\dot x$ decreases. Therefore we have that $g(t)$ converges to a metric $g(T)$ as $t\to T$  and that we can extend the solution $g(t)$ afterward $T$, contradicting $T<\infty$.  
\end{proof}

\begin{rem}{\em 
It is rather natural asking what happens in the proof of Theorem \ref{lte} when the initial metric $g_o$ takes the diagonal expression $g_o=x_o\alpha^{1\bar 1}+y_o\alpha^{2}\alpha^{\bar 2}+z_o\alpha^{3}\alpha^{\bar 3}.$ 
For $\g_2^0$ any diagonal metric is K\"ahler Ricci-flat and as such does not evolve with the pluricosed flow. In the case of $\g_4$, $g_o$ evolves as  
$$
g(t)=x(t)\alpha^{1}\alpha^{\bar1}+y_o\alpha^{2}\alpha^{\bar2}+z_o\alpha^{3}\alpha^{\bar3}
$$
where the component $x$ solves 
$$
\dot x=2 \frac{z_o}{x}
$$
and so $x$ takes the following form
$$
x=\sqrt{2x_o+4z_ot}
$$
which is defined for $t\geq -\frac{x_o}{2z_o}$. }
\end{rem}

\section{Static metrics and Hermitian-Symplectic structures}
In \cite{streets-tian2} Streets and Tian introduced the definition of {\em static metric} as a natural generalization of K\"ahler-Einstein metrics to the SKT setting.  
\begin{definition}[Streets-Tian \cite{streets-tian2}]
An SKT metric $g$ with fundamental form $\omega$ on a complex manifold $(M,J)$ is called {\em static} if $\rho^B(g)=\lambda \omega$,  for a constant $\lambda\in\R$.  
\end{definition}
An example of a non-K\"ahler compact complex manifold carrying a static metric with $\lambda=0$ is provided by the  {\em Hopf surface} $S^3 \times S^1$. Currently it is not known any example of a compact complex non-K\"ahler manifold carrying a static metric with $\lambda\neq 0$. Indeed, the existence of a static metric with $\lambda\neq 0$ imposes some restrictions; one of them is the existence of a symplectic form $\Omega$ taming the complex structure. 
More precisely, if $g$ is a static SKT metric on $(M,J)$ with $\lambda\neq 0$, then $\Omega=\frac{1}{\lambda}\rho^B$ is a symplectic form on $(M,J)$ such that 
$\Omega(X,JX)>0$, for every non-zero vector field $X$.

We recall the following 
\begin{definition}
Let $(M,J)$ be an almost complex manifold. A symplectic structure $\Omega$ on $M$ {\em tames} $J$ if 
\begin{equation*}\label{tames}
\Omega(X,JX)>0,
\end{equation*}
for every non-zero vector field $X$ on $M$. If  in addition 
$$
\Omega(JX,JY)=\Omega(X,Y),
$$ 
for every vector fields $X,Y$ on $M$, then $\Omega$ is {\em compatible} with $J$.  If $J$ is integrable and $\Omega$ is a taming symplectic form,  the pair $(\Omega,J)$ is called a {\em Hermitian-symplectic} structure. 
\end{definition}
Therefore the existence of an SKT static metric with $\lambda\neq 0$ implies the existence of an Hermitian-symplectic structure and the existence of  Hermitian-symplectic structure is an obstruction to the existence of a static metric with $\lambda\neq 0.$ 

\begin{problem}[Streets-Tian \cite{streets-tian2}]
Find examples of compact Hermitian-symplectic manifolds   non-admitting K\"ahler metrics. 
\end{problem}
About this problem there are some negative results in literature which suggest that Hermitian-symplectic structures on non-K\"ahler manifolds couldn't exist. The first of these results is about the four dimensional case.
\begin{theorem}[Li-Zhang \cite{lizhang}, Streets-Tian \cite{streets-tian2}]
If a compact complex surface admits an Hermitian-symplectic form, then it is K\"ahler. 
\end{theorem}
In \cite{EFV} it is studied the existence of an Hermitian-symplectic structure  when $(M,J)$ is a nilmanifold with an invariant complex structure. Since nilmanifolds carry both complex and symplectic structures it is rather natural to explore the existence of an Hermitian-symplectic structure in this class of examples.  

\begin{theorem}[Enrietti-Fino-Vezzoni \cite{EFV}]\label{EFV}
An invariant complex structure $J$ on a {nilmanifold} $M$ can be tamed by a symplectic form if and only if $(M,J)$ is a complex torus.
\end{theorem}
The proof of Theorem \ref{EFV} makes use of the following lemma which is interesting in itself.
\begin{lemma}
Let $(M=\Gamma\backslash G ,J,g)$ be a nilmanifold with an invariant SKT structure. Then the Lie algebra of $G$ is abelian or $2$-step and $J$ preserves its center. 
\end{lemma}

As a direct application of Theorem \ref{EFV} we have the following

\begin{cor}
Let $M=\Gamma\backslash G$ be a nilmanifold together with an invariant complex structure $J$. Then $M$ does not admit any $J$-Hermitian invariant static metric with $\lambda\neq 0$ unless it is a complex torus.
\end{cor}
About the problem of the existence of a static metric with $\lambda=0$ on a nilmanifold we have the following  

\begin{theorem}[Enrietti \cite{enrietti}]
Let $M=\Gamma \backslash G$ a nilmanifold together with an invariant complex structure $J$. Then $M$ does not admit any $J$-Hermitian invariant static metric with 
$\lambda=0$ unless it is a complex torus. 
\end{theorem}
Since  a static metric  on a complex manifold  induces a symplectic structure taming the complex structure,  it follows that  a nilmanifold  equipped with an invariant complex structure $J$ cannot admit a non-invariant static metric having $\lambda\neq 0$, unless $M$ is a complex torus.  It  would be interesting  to extend Theorem \ref{EFV} to the almost complex case. 

\begin{problem}\label{tamed}
Let $(M,J)$ be a nilmanifold with an {\em invariant } almost complex structure. Does the existence of a symplectic form taming $J$ imply the existence of an invariant  symplectic form compatible with $J$?
\end{problem}

Problem \ref{tamed} was confirmed in \cite{litomassini}  for  the Kodaira-Thurston manifold by Li and Tomassini. 
Some partial results  about the existence of Hermitian-symplectic structures on solvmanifolds $M = \Gamma \backslash G$ endowed with an invariant complex structure $J$ have been obtained in \cite{conilgiapu}, showing that if either $J$ is invariant under the action of a nilpotent complement of the nilradical of  $G$  or $J$ is abelian or $G$  is almost abelian (not of type (I)), then the solvmanifold $\Gamma \backslash G$ cannot admit any symplectic form taming the complex structure $J$, unless  $\Gamma \backslash G$ is K\"ahler.  In particular,  the family of non-K\"ahler complex manifolds constructed by Oeljeklaus and Toma \cite{OT}  cannot admit any symplectic form taming the complex structure.

By \cite{EFG} it turns out that symplectic forms taming complex structures on compact manifolds are related to special types of almost generalized K\"ahler structures. Indeed, by considering the commutator  $Q$  of the two associated almost complex structures  $J_{\pm}$,  it is shown that if either the manifold is 4-dimensional or the distribution $Im (Q)$  is involutive, then the manifold can be expressed locally as a disjoint union of twisted Poisson leaves. It would be interesting to see if this property can be extended in higher dimensions.

\section{SKT and balanced structures}
Another important class of Hermitian metrics is provided by {\em balanced metrics}. An Hermitian metric on a complex manifold $(M,J)$ is called {\em balanced} if its fundamental form $\omega$ co-closed or equivalently if its Lee form $\theta$ vanishes.  By \cite{aliv}  in  real dimension $2n \geq  6$   the vanishing of $\theta$  is complementary to the SKT condition, i.e.  an Hermitian metric which is simultaneously balanced and SKT  has to be K\"ahler.
Balanced  structures were characterised  in terms of currents by Michelshon  \cite{M},  where a deep obstruction for the existence of a such metrics   is provided. From Michelshon's paper it  in particular  that Calabi-Eckmann manifolds have no balanced metrics. Typically examples of complex manifolds admitting a balanced metric are twistor spaces of compact anti-self-dual $4$-dimensional Riemannian manifolds.

About the existence of SKT and balanced metrics we propose the following problem. 

\begin{problem}\label{balanced}  Show that a  compact complex manifold $(M,J)$ cannot admit a compatible SKT metric $g$ and also a compatible balanced metric $\tilde g$ unless $(M, J) $ is K\"ahler. 
\end{problem}

The above problem has been implicitly  already solved  in literature  in some special cases. For instance: Verbitsky proved in \cite{verb} that the twistor space of
a compact, anti-self-dual Riemannian manifold has no SKT metrics unless it has K\"ahler metrics and 
Chiose proved  in \cite{chiose}  a similar result for non-K\"ahler manifolds belonging to the Fujiki class. Moreover, Li, Fu and Yau found in \cite{fuliyau} a new class of non-K\"ahler balanced manifolds by using conifold transactions. Such examples include the connected sums $M_k$ of $k$-copies of $S^3\times S^3$, $k\geq 1$. It is proved in \cite{fuliyau} that $M_k$ has no SKT metrics. 

A restriction can be given in terms of the  Bott-Chern cohomology groups, which are defined  for a  general complex manifold $(M,J)$ as
$$
H^{p,q}_{BC} (M) =\frac{\left\{\alpha \in \Omega^{p,q}(M)\,\,:\,\,d\alpha=0\right\}}{\left\{\partial \bar\partial\gamma\,\,:\,\,\gamma \in \Omega^{p-1,q-1}(M)\right\}}\,. 
$$
\begin{prop}
Let $(M,J)$ be compact complex manifold having $H^{n-1,n-1}_{BC} (M)=0$. If $(M,J)$ has a balanced metric, then it has no SKT metrics.    
\end{prop}
\begin{proof}
Assume that $(M,J)$ admits an SKT metric $g$ and also a balanced metric $\tilde g$ and let $\omega$ and $\tilde \omega$ be the induced fundamental forms.  Then $\omega\wedge \tilde \omega^{n-1}$ is a volume form on $M$ and so 
$$
\int_M\omega\wedge \tilde \omega^{n-1}\neq 0\,. 
$$
In particular if $H_{BC}^{n-1,n-1}(M,J)=0$, then $\tilde \omega^{n-1}=dd^c \alpha$ for some $(n-2,n-2)$-form $\alpha$ and then 
$$
\int_M\omega\wedge \tilde \omega^{n-1}=\int_M\omega\wedge dd^c\alpha=\pm \int_M dd^c\omega \wedge \alpha=0 
$$
which is a contradiction.
\end{proof}
 
 \subsection{Problem \ref{balanced} for  nilmanifolds and solvmanifolds}
 In this last section we  study  Problem \ref{balanced}  for   nilmanifolds  of dimension $6$ or $8$  and for  $6$-dimensional solvmanifolds. For the nilpotent case we assume that the complex structure $J$  is invariant and in the solvable case we  suppose that  in addition  $J$ has the canonical bundle holomorphically trivial. 
 
 First of all we recall the following result  (see \cite{finograntharov,ugarte}) which allows us to assume the metrics to be  invariant and to work at the level of the Lie algebra of $G$.
\begin{theorem}\label{anna+gueo}
Let $M=\Gamma\backslash G$ be the compact quotient of a Lie group $G$ by a discrete subgroup $\Gamma$  equipped with an invariant complex structure $J$. Then 
\begin{enumerate}
\item[-] $M$ has an SKT metric if and only if it has an {\em invariant} SKT metric;  
\item[-] $M$ has a balanced metric if and only if it has an {\em invariant} balanced metric.
\end{enumerate}
\end{theorem}

The next theorem is about Problem \ref{balanced} when $M$ is a nilmanifold of dimension $6$ or $8$.

\begin{theorem}
Let $M=\Gamma\backslash G$ be a   nilmanifold  equipped with an invariant complex structure $J$.   If   $M$ has dimension 6 or 8, then  $(M, J)$ cannot admit a compatible SKT metric $g$ and also a compatible balanced metric $\tilde g$ unless  $(M, J) $ is  K\"ahler, i.e.  $G$ is abelian. 
\end{theorem}
 \begin{proof} 
By  Theorem \ref{anna+gueo}  we may assume that  $g$ and $\tilde g$ are both  invariant.  Then we can suppose that  $(M,J)$ has an invariant SKT metric.  If $\dim M = 6$,   the existence of an
SKT structure on a nilpotent Lie algebra $\mathfrak g$  depends only on the complex structure of  $\mathfrak g$. Indeed, Theorem \ref{finosalamonparton} implies that also $g$ has to be SKT. Therefore $g$ is simultaneously balanced and SKT and hence it is K\"ahler.

 If $\dim M = 8$, the situation is more complicated and we can use a  classification obtained in \cite{EFV}. More precisely, by \cite{EFV}
the existence of an {\rm SKT} metric $g$ on $M$ compatible with $J$ implies that the Lie algebra $(\g,J)$ has a $(1,0)$-coframe 
$\{\alpha^1,\alpha^2,\alpha^3,\alpha^4\}$ satisfying one of the following structure equations 
\begin{enumerate}
\item[\em{1.}] {{\em First family}}:
\begin{equation} \label{structeq1}
\begin{cases}
d \alpha^j= 0, \quad j = 1,2,\\
d \alpha^3 = B_1 \alpha^{1 2} + B_4 \alpha^{1 \overline 1} + B_5 \alpha^{1 \overline 2} + C_3 \alpha^{2 \overline 1} + C_4 \alpha^{2 \overline 2}\,,\\
d \alpha^4 =F_1\alpha^{12} + F_4 \alpha^{1 \overline 1} + F_5 \alpha^{1 \overline 2} + G_3 \alpha^{2 \overline 1} + G_4 \alpha^{2 \overline 2}\,,
\end{cases}
\end{equation}
where   the capital letters are arbitrary complex numbers;
\item[{\em 2.}] {\em Second family}: 
\begin{equation} \label{structeq2}
 \begin{cases}
  \begin{aligned}
 d \alpha^j =&0, \quad j = 1,2,3\,,\\
 d \alpha^4 = &F_1 \alpha^{12} + F_2 \alpha^{13}+ F_4 \alpha^{1 \overline 1} + F_5 \alpha^{1 \overline 2} + F_6 \alpha^{1 \overline 3} + G_1 \alpha^{23} +  G_3 \alpha^{2 \overline 1} + G_4 \alpha^{2 \overline 2} \\[-3pt]
 &+ G_5 \alpha^{2 \overline 3}
 + H_2 \alpha^{3 \overline 1} + H_3 \alpha^{3 \overline 2} + H_4 \alpha^{3 \overline 3}\,,
\end{aligned}
\end{cases}
\end{equation}
where  the capital letters are arbitrary complex numbers and  $H_4\neq 0$.
\end{enumerate}
For the second family,  the SKT equations  for a generic  $J$-Hermitian metric  $g$ are 
$$
\left \{ \begin{array}{l}
- H_3 \overline F_4  + H_2 \overline G_3 + F_5 \overline F_6 - F_4  \overline G_5 + F_2 \overline F_1=0,\\[4 pt]
-H_3 \overline F_5  + G_4 \overline F_6  + H_2  \overline G_4 - G_3  \overline G_5 + G_1 \overline F_1 =0,\\[4 pt]
-H_4 \overline F_5 + G_5 \overline F_6 + H_2 \overline H_3 -G_3  \overline H_4 + G_1 \overline F_2=0,\\[4 pt]
\vert F_2 \vert^2 + \vert F_6 \vert^2 +\vert H_2 \vert^2= 2 \, {\mbox {Re}} (H_4 \overline F_4),\\[4 pt]
 \vert F_1 \vert^2 + \vert F_5 \vert^2 +  \vert G_3 \vert^2 = 2 \, {\mbox {Re}} (F_4 \overline G_4),\\[4 pt]
  \vert G_1 \vert^2   + \vert G_5 \vert^2 + \vert H_3 \vert^2 = 2 \, {\mbox {Re}} (H_4 \overline G_4 )
\end{array} \right.
$$
  and so as in the $6$-dimensional case the SKT condition depends only on the complex structure and the theorem follows.
  
 For the first  family it is not anymore true that the existence of an
SKT metric    depends only on the complex structure.
Indeed, consider a generic  $J$-Hermitian metric $g$. The   fundamental form  $\omega$  associated to the Hermitian structure $(J, g)$ can be then expressed as $$
\begin{aligned}
\omega = & a_1 \alpha^{1\ov{1}} + a_2 \alpha^{2\ov{2}} + a_3 \alpha^{3\ov{3}} + a_4 \alpha^{4\ov{4}} +a_5  \alpha^{1 \overline 2}- \overline a_5 \alpha^{2 \overline 1} +a_6 \alpha^{1 \overline 3} -   \overline a_6 \alpha^{3 \overline 1} + a_7 \alpha^{1 \overline 4} -\overline  a_7 \alpha^{4 \overline 1} \\[-3pt]
& +  a_8 \alpha^{2 \overline 3} - \overline a_{8}\alpha^{3 \overline 2}  +a_9 \alpha^{2 \overline 4} - \overline a_{9} \alpha^{4 \overline 2} +  a_{10}  \alpha^{3\ov{4}} - \ov{a}_{10} \alpha^{4\ov{3}}\,,
\end{aligned}
$$
where  $a_l$, $l = 1, \dots, 10$, are arbitrary   complex numbers     (with  $\overline a_l = - a_l$, for any $l = 1, \ldots, 4$) such that  $\omega$  is positive definite.
The SKT equation for a generic $J$-Hermitian metric $g$ is:
\begin{equation}\label{SKTanna}
\begin{aligned}
&- a_3 C_4 \overline B_4 - 2 a_{10} B_4 \overline G_4 - a_3 B_4 \overline C_4 + a_{10} B_1 \overline F_1 + a_3 \vert B_1 \vert^2 - \overline a_{10} \overline B_1 F_1 + \overline a_{10} \overline C_4 F_4\\[-3pt]
 &+ a_4 \vert F_1 \vert^2 - \overline a_{10} G_3 \overline C_3 + a_4 \vert G_3 \vert^2 + a_3 \vert B_5 \vert^2 - \overline a_{10} \overline B_5 F_5 - a_4 \overline G_4 F_4 + a_3 \vert C_3 \vert^2\\[-3pt]
  &+ a_4 \vert F_5 \vert^2 + a_{10} \overline F_5 B_5 - a_4 \overline F_4 G_4 + \overline a_{10} G_4 \overline B_4
   + a_{10} \overline G_3 C_3 - a_{10} C_4 \overline F_4 =0,
\end{aligned}
\end{equation}
so it not anymore true  that   that every $J$-Hermitian metric is ${\rm SKT}$.

We can show that  the  nilpotent Lie algebras  of the first family  admit  balanced metrics if and only if they have a coframe $\{\alpha^1,\alpha^2,\alpha^3,\alpha^4\}$  of $(1,0)$-forms satisfying \eqref{structeq1} with 
$C_4 = - B_4$ and $G_4 = - F_4$, i.e.
\begin{equation}  \label{newbasis}
\begin{cases}
d \alpha^j= 0, \quad j = 1,2,\\
d \alpha^3 = B_1 \alpha^{1 2} + B_4 \alpha^{1 \overline 1} + B_5 \alpha^{1 \overline 2} + C_3 \alpha^{2 \overline 1} -B_4 \alpha^{2 \overline 2}\,,\\
d \alpha^4 =F_1\alpha^{12} + F_4 \alpha^{1 \overline 1} + F_5 \alpha^{1 \overline 2} + G_3 \alpha^{2 \overline 1} -F_4 \alpha^{2 \overline 2}\,.
\end{cases}
\end{equation}
Moreover, the  Hermitian metric  associated to  $\omega = -i\sum_{j = 1}^4  \alpha^{j}\wedge\alpha^{ \overline j}$  is always balanced. 

Indeed, applying the  Gram-Schmidt process to the the basis  $\{  \alpha^1, \dots, \alpha^4 \}$  satisfying the structure equations  \eqref{structeq1} we get  a $g$-unitary coframe
$\{  \tilde \alpha^1, \dots, \tilde \alpha^4 \}$ such that
$$
\left \{ \begin{array}{l}
d \tilde  \alpha^j = 0, \quad j = 1, 2,\\
d \tilde \alpha^l \in \Lambda^2 \langle \tilde \alpha^1, \ldots, \tilde \alpha^p, \overline {\tilde  \alpha^1}, \ldots , \overline {\tilde \alpha^p} \rangle, \quad  l = 3, 4,
\end{array} \right.
$$
since ${\mbox {span}} \langle  \tilde \alpha^1, \ldots,  \tilde \alpha^j\rangle={\mbox {span}} \langle \alpha^1, \ldots, \alpha^j\rangle $, for any $j= 1, \ldots, 4$. Then it is not restrictive to assume that  the basis $ \{\alpha^1,\dots,\alpha^4\}$
is $g$-unitary, i.e.  that  the fundamental form $\omega$ of $g$ with respect to $\{\alpha^1,\dots,\alpha^4\}$ takes the standard expression:
$$
\omega = -i\sum_{j = 1}^4  \alpha^{j}\wedge\alpha^{ \overline j}.
$$
Finally a direct computation implies that  $d \omega^3 =0$ if and only if  $C_4 = - B_4$, $G_4 = - F_4$.

Now,  it is possible to prove that the nilpotent Lie algebras  $\frak g$  with the $(1,0)$-coframe $\{\alpha^k\}$ satisfying \eqref{newbasis} cannot  have    SKT metrics, unless  $\frak g$  is abelian.

Assume by contraddiction that  $\frak g$  admits an SKT metric $g$  and let $\omega$ be the associated 
fundamental form.   We may write  $\omega$   as $$
\begin{aligned}
\omega = & a_1 \alpha^{1\ov{1}} + a_2 \alpha^{2\ov{2}} + a_3 \alpha^{3\ov{3}} + a_4 \alpha^{4\ov{4}} +a_5  \alpha^{1 \overline 2}- \overline a_5 \alpha^{2 \overline 1} +a_6 \alpha^{1 \overline 3} -   \overline a_6 \alpha^{3 \overline 1} + a_7 \alpha^{1 \overline 4} -\overline  a_7 \alpha^{4 \overline 1} \\[-3pt]
& +  a_8 \alpha^{2 \overline 3} - \overline a_{8}\alpha^{3 \overline 2}  +a_9 \alpha^{2 \overline 4} - \overline a_{9} \alpha^{4 \overline 2} +  a_{10}  \alpha^{3\ov{4}} - \ov{a}_{10} \alpha^{4\ov{3}}\,,
\end{aligned}
$$
where  $a_l$, $l = 1, \dots, 10$, are arbitrary   complex numbers     (with  $\overline a_l = - a_l$, for any $l = 1, \ldots, 4$)  satisfying 
\begin{equation}\label{condSKTbal}
\begin{array}{c}
a_3  (2 \vert B_4 \vert^2   + \vert B_1 \vert^2 +  \vert B_5 \vert^2 + \vert C_3 \vert^2)  + a_4 ( \vert F_1 \vert^2 +    \vert G_3 \vert^2   +2   \vert F_4 \vert^2 +  \vert F_5 \vert^2 ) = \\[3pt]
- a_{10} \overline F_5 B_5   + \overline a_{10} \overline B_5 F_5
   -a_{10} \overline G_3 C_3  + \overline a_{10} G_3 \overline C_3  -  2 a_{10} B_4 \overline F_4 + 2 \overline a_{10} \overline B_4 F_4 - a_{10} B_1 \overline F_1 + \overline a_{10} \overline B_1 F_1
\end{array}
\end{equation}
and such that  $\omega$  is positive definite.

Condition \eqref{condSKTbal} can be rewritten as 
$$
\omega(X_1,\bar X_1)+\omega(X_2,\bar X_2)+\omega(X_3,\bar X_3)+\omega(X_4,\bar X_4)=0
$$
where 
$$
\begin{aligned}
&X_1=B_1 Z_3 + F_1 Z_4\,,\quad X_2=B_5 Z_3 + F_5 Z_4\,,\\
&X_3=C_3 Z_3 + G_3 Z_4\,,\quad X_4=\sqrt{2} B_4 Z_3 + \sqrt{2} F_4 Z_4
\end{aligned}
$$
and $\{ Z_j \}$ is dual  frame dual  of  $\{ \alpha^j \}.$ Since $\omega$ is positive definite, we have $X_1 = X_2 = X_3= X_4=0$ which implies that all the forms $\alpha^k$'s must  be closed, i.e.  that $\frak g$ is abelian. 
\end{proof}

Now we treat the solvable case.

\begin{theorem}
Let $M=\Gamma\backslash G$ be a  $6$-dimensional  solvmanifold  equipped with an invariant complex structure $J$ with holomorphically trivial canonical bundle.   Then  $(M, J)$ cannot admit a compatible SKT metric $g$ and also a compatible balanced metric $\tilde g$ unless  $(M, J)$ is  K\"ahler. 
\end{theorem}
 \begin{proof} 
Suppose that $\frak g$ is not nilpotent. By  \cite[Theorem 4.5]{FU}  if $ (M, J)$  has a balanced metric, then the Lie algebra of $\frak g$ is isomorphic to one of the following Lie algebras:
$$
\begin{array}{l}
\frak g_1 =  (e^{15},-e^{25},-e^{35}, e^{45}, 0, 0), \\[3pt]
\frak g_2^{\alpha} = (\alpha e^{15}+e^{25},-e^{15}+\alpha e^{25},- \alpha e^{35} +e^{45},-e^{35}- \alpha e^{45}, 0, 0),  \, \alpha \geq 0, \\[3pt]
\frak g_3 =  (0,-e^{13}, e^{12}, 0,-e^{46},-e^{45}),\\[3pt]
\frak g_5=  (e^{24} + e^{35}, e^{26}, e^{36},-e^{46},-e^{56}, 0), \\[3pt]
\frak g_7 = (e^{24} + e^{35}, e^{46}, e^{56},-e^{26},-e^{36}, 0), \\[3pt]
\frak g_8 =  (e^{16} - e^{25}, e^{15} + e^{26},-e^{36} + e^{45},-e^{35} - e^{46}, 0, 0).
\end{array}
$$
On the other hand by \cite[Theorem 4.1]{FU}
$(M,J)$ has an SKT metric if and only if   the Lie algebra $\frak g$ is either isomorphic to $\frak g_2^0$ or $\frak g_4$. Therefore if  $(M, J)$ admits a   $J$-Hermitian balanced metric and a $J$-Hermitian  SKT metric, then $\frak g$ has to be abelian or  isomorphic to  $\frak g_2^0$. By \cite{FU} every complex structure $J$  on $\frak g \cong  \frak g_2^0$ with a closed $(3,0)$-form is equivalent to  the complex structure $J$  defined by 
the structure equations
\b\label{streqs}
d \alpha^1 = i  (\alpha^{13} + \alpha^{1 \overline 3}), \, d \alpha^2 = - i \alpha^{23} - i  \alpha^{2 \overline 3}, d \alpha^3 =0
\e
with respect to a suitable $(1,0)$-coframe $(\alpha^k)$.  Moreover $(\frak g_2^0, J)$ admits the K\"ahler  metric with associated fundamental form $$\omega = i \alpha^{1\ov 1} + i \alpha^{2 \ov 2} + i \alpha^{3\ov 3}
$$
and so $(M, J)$ is K\"ahler.
\end{proof}

\end{document}